\documentclass[preprint,12pt]{elsarticle}

\usepackage{amssymb}
\usepackage{graphicx}
\usepackage{mathptmx}      

\usepackage{amssymb,amsfonts,amsmath,amscd,latexsym}
\usepackage[english]{babel}
\usepackage{geometry}
\usepackage{latexsym}
\usepackage{amssymb}
\usepackage[all]{xy}
\usepackage{amssymb}
\usepackage{latexsym}
\usepackage{graphicx}
\usepackage{appendix}
\usepackage[active]{srcltx}
\usepackage{srcltx}
\usepackage{xypic}
\usepackage{amsthm}

\newdefinition{defn}{Definition}[section]
\newdefinition{nim}[defn]{}
\newdefinition{rem}[defn]{Remark}
\newdefinition{ex}[defn]{Example}
\newtheorem{thm}{Theorem}[section]
\newtheorem{cor}[thm]{Corollary}
\newtheorem{lem}[thm]{Lemma}
\newtheorem{prop}[thm]{Proposition}

\newcommand{\tr}{\mathrm{tr}}

\newcommand{\Hc}{\mathrm{H}}
\newcommand{\Hs}{\mathrm{HS}}

\newcommand{\Ker}{\mathrm{Ker}}
\newcommand{\Ima}{\mathrm{Im}}

\newcommand{\res}{\mathrm{res}}

\begin{document}

\begin{frontmatter}



\title{Symmetric cohomology of groups as a Mackey functor}


\author{Constantin-Cosmin Todea}

\address{Department of Mathematics, Technical University of Cluj-Napoca, Str. G. Baritiu 25,
 Cluj-Napoca 400027, Romania}

\ead{Constantin.Todea@math.utcluj.ro}
\begin{abstract} Symmetric cohomology of groups, defined by M. Staic in \cite{St3D}, is similar to the way one defines the cyclic cohomology for algebras. We show that there is a well-defined restriction, conjugation and transfer map in symmetric cohomology, which form a Mackey functor under a restriction. Some new properties for the symmetric cohomology group using normalized cochains are also given.
\end{abstract}


\end{frontmatter}


\section{Introduction}\label{sec1}
Symmetric cohomology for groups was introduced by M. Staic in \cite{St3D} in order to associate to topological spaces some elements in the third symmetric cohomology of some groups. Further algebraic properties of symmetric cohomology of groups in low dimension were studied by the same author in \cite{Starc}. M. Singh also studies symmetric continuous cohomology of topological groups in \cite{Sin}. He shows that the symmetric continuous cohomology of a profinite group with coefficients in a discrete module is equal to the direct
limit of the symmetric cohomology of finite groups.  In Section \ref{sec2} we continue the investigation of the first symmetric cohomology of a group and we propose a new approach for defining symmetric cohomology (with coefficients in a particular $G$-module, where $G$ is a group) using normalized cochains, which seems to give some easier conditions to be verified. In Section \ref{sec3} we prove the main result of this paper, Theorem \ref{thmmackey} where we verify that the restriction, conjugation and transfer in symmetric cohomology satisfy the axioms of  a Mackey functor. Some axioms are verified in general but for a few we need a restriction. If we can drop this restriction remains an open problem, which is proposed in Remark \ref{res}. We believe that these results should be important for further developments in algebraic and topological context.

For the rest of this section we recall the definition of symmetric cohomology of groups, some notations and definitions for group cohomology. First we recall some well-known facts about ordinary group cohomology: restriction, transfer and conjugation map. We will give the explicit description of these maps using standard cochains from \cite{We}. Let $G$ be a group, let $A$ be a $G$-module and $n\geq 0$ an integer. For $G/H$ the set of left cosets of $H$ in $G$, also denoted $G=\bigcup_{c\in G/H}c$, choose once and for all a representative $\overline{c}\in c$. By convention if $c=H$ we require $\overline{c}=1$. If $g_1,\ldots,g_n$ are $n$ elements from $G$ we will use the notations
$$x_1=g_1\ldots g_n,~x_2=g_2\ldots g_n,~\ldots,~x_n=g_n.$$
Recall that the abelian group of $n$-cochains is $C^n(G,A)=\{\sigma:G^n\rightarrow A\}$ and  we define the differential $\partial^n:C^n(G,A)\rightarrow C^{n+1}(G,A)$ by
$$\partial^n(\sigma)(g_1,\ldots,g_{n+1})=$$$$g_1\sigma(g_2,\ldots,g_{n+1})+\sum_{i=1}^n(-1)^i\sigma(g_1,\ldots,g_ig_{i+1},\ldots,g_{n+1})+(-1)^{n+1}\sigma(g_1,\ldots,g_{n}).$$
The homology of the cochain complex $(C^*(G,A),\partial^*) $  is called the cohomology of $G$ with coefficients in $A$
$$\Hc^n(G,A):=\Ker\partial^n/\Ima\partial^{n-1}.$$
Let $H$ be a subgroup of $G$ and $g\in G$. Sometimes, when more groups are involved, for explicitness we will denote the differential with an index. For example $\partial_H,\partial_{{}^gH}$ are used for cohomology of the group $H$, respectively of the group ${}^gH$, where ${}^gH=gHg^{-1}$. By \cite[Proposition 2.5.1]{We} we have:
\begin{nim}\label{res}\emph{Restriction.} $$\res^G_H:\Hc^n(G,A)\rightarrow\Hc^n(H,A),~~~\res^G_H[\sigma]=[\res^G_H(\sigma)],$$
where $\res^G_H(\sigma)(h_1,\ldots,h_n)=\sigma(h_1,\ldots,h_n)$ for any $h_1,\ldots,h_n\in H$ and $\sigma\in\Ker\partial_G^n$.
\end{nim}
\begin{nim}\label{conj}\emph{Conjugation.} $$c_{g,H}:\Hc^n(H,A)\rightarrow\Hc^n(^gH,A),~~~c_{g,H}[\sigma]=[c_{g,H}(\sigma)],$$
where $c_{g,H}(\sigma)(^gh_1,\ldots,{}^gh_n)=g\sigma(h_1,\ldots,h_n)$ for any $h_1,\ldots,h_n\in H$ and $\sigma\in\Ker\partial_H^n$.
\end{nim}
\begin{nim}\label{tran}\emph{Transfer.} $$\tr_H^G:\Hc^n(H,A)\rightarrow\Hc^n(G,A),~~~\tr_H^G[\sigma]=[\tr_H^G(\sigma)],$$
where $$\tr_H^G(\sigma)(g_1,\ldots,g_n)=\sum_{c\in G/H}\overline{x_1c}\sigma(\overline{x_1c}^{-1}g_1\overline{x_2c},\overline{x_2c}^{-1}g_2\overline{x_3c},\ldots,\overline{x_nc}^{-1}g_n\overline{c})$$ for any $g_1,\ldots,g_n\in G$ and $\sigma\in\Ker\partial_H^n$. Notice that $\overline{x_ic}^{-1}g_i\overline{x_{i+1}c}\in H$ and $\overline{x_nc}^{-1}g_n\overline{c}\in H$ for any $i\in\{1,\ldots,n-1\}$.
\end{nim}

In \cite{St3D} M. Staic defines an action of $\Sigma_{n+1}$ (the symmetric group on $n+1$ letters) on $C^n(G,A)$, using a generating set of transpositions $\{(1,2);(2,3);\ldots;(n,n+1)\}$ by
\begin{nim}\label{act}
$$((1,2)\sigma)(g_1,\ldots,g_n)=-g_1\sigma(g_1^{-1},g_1g_2,g_3,\ldots,g_n),$$
$$((i,i+1)\sigma)(g_1,\ldots,g_n)=-\sigma(g_1,\ldots,g_{i-1}g_i,g_i^{-1},g_ig_{i+1},\ldots,g_n)~~~\text{for}~1<i<n,$$
$$((n,n+1)\sigma)(g_1,\ldots,g_n)=-\sigma(g_1,g_2,\ldots,g_{n-1}g_n,g_n^{-1}).$$
\end{nim}
Now \cite[Proposition 5.1]{St3D} assure us that formulas \ref{act} give a well-defined action of $\Sigma_{n+1}$ on $C^n(G,A)$ compatible with the differentials $\partial$. Hence we have a new cohomology.

\begin{defn}\cite[Definition 5.2]{St3D}\label{defsymmcohom} The subcomplex of invariants denoted $CS^n(G,A)=C^n(G,A)^{\Sigma_{n+1}}$ is called the \emph{symmetric cochain complex}. Its homology is the \emph{symmetric cohomology} of $G$ with coefficients in $A$ and is denoted $\Hs^n(G,A)=ZS^n(G,A)/BS^n(G,A).$
\end{defn}
\section{Some remarks on $\Hs^1(G,A)$ and $\Hs^2(G,A)$.}\label{sec2}
In \cite{Starc} the author gives conditions for the natural map $\Hs^n(G,A)\rightarrow\Hc^n(G,A)$ to be injective. For the first and the second symmetric cohomology these natural maps are injective in general. Similar results are obtained in \cite{Sin} for (symmetric) continuous cohomology of topological groups, denoted ($\Hs_c^n(G,A)$) $\Hc_c^n(G,A)$. We will show the easy detail, missed in both papers, that for the first symmetric cohomology  group we actually have an equality.
\begin{prop}\label{HS1} With the above notations we have an equality $\Hs^1(G,A)=\Hc^1(G,A)$. In particular, in the context of continuous symmetric topological groups (\cite[Section 3]{Sin}) we have $\Hs_c^1(G,A)=\Hc_c^1(G,A)$.
\end{prop}
\begin{proof}
Recall the well-known fact $\Hc^1(G,A)=Der(G,A)/Pder(G,A)$, where $$Der(G,A)=\{\sigma:G\rightarrow A\mid \sigma(gh)=g\sigma(h)+\sigma(g)\},$$
$$Pder(G,A)=\{\sigma_a:G\rightarrow A\mid a\in A,~\sigma_a(g)=ga-a\}$$
One can see, from Definition \ref{defsymmcohom}, that $\sigma\in CS^1(G,A)$ if and only if $ \sigma(g)=-g\sigma(g^{-1})$. Therefor $\sigma\in ZS^1(G,A)$ if and only if $\sigma(gh)=g\sigma(h)+\sigma(g)$ and $\sigma(g)=-g\sigma(g^{-1})$. We will prove next that $ZS^1(G,A)=Der(G,A)$. The inclusion from left to right is trivial. Let $\sigma\in Der(G,A)$ and take $g=h=1$ in the derivation relation to obtain $\sigma(1)=\sigma(1)+\sigma(1)$, thus $\sigma(1)=0$. Now take $h=g^{-1}$ in the same derivation relation to get
$$\sigma(1)=g\sigma(g^{-1})+\sigma(g).$$
Since $\sigma(1)=0$ we obtain that $\sigma\in ZS^1(G,A)$.

One can also see from Definition \ref{defsymmcohom} that $CS^0(G,A)=A$; hence the definition of $\partial^0$ assure us that $Pder(G,A)=BS^1(G,A)$.
\end{proof}
It is known that if we use normalized cochains the group cohomology is the same.  We will show that by using normalized cochains the conditions on a symmetric cocycle and on a symmetric boundary are much easier to handle. Let $n\geq 0$ be an integer. Recall that $\Hc^n(G,A)=Z^n(G,A)/B^n(G,A)$, where $\sigma \in Z^n(G,A)$ is a normalized cocycle if $\partial_G^n(\sigma)=0$ and
$$\sigma(g_1,\ldots,g_n)=0$$
for any $g_1,\ldots,g_n\in G$ with some $g_i=1$ where $i\in\{1,\ldots,n\}$.
Similarly we kave the definition of a normalized boundary.
We say that a $n$-cochain $\sigma:G^n\rightarrow A$ satisfies the $(\ast)$ conditions if
$$\sigma(g_1,g_1^{-1},g_3,\ldots,g_n)=0$$
$$\sigma(g_1,\ldots,g_{i-1},g_i,g_i^{-1},g_{i+2},\ldots,g_n)=0~~~\text{for}~1<i<n$$
$$\sigma(g_1,\ldots,g_{n-2},g_n,g_n^{-1})=0$$
for any $g_1,\ldots,g_n\in G$. Also, we say that a $n$-boundary $\sigma:G^n\rightarrow A$ is a \emph{$n$-$(\ast)$-boundary} if there is $\beta:G^{n-1}\rightarrow A$ which satisfies the corresponding $(\ast)$ conditions such that $\partial^n(\beta)=\sigma$.  For the rest of this section we will we work with normalized $n$-cochains in the case of usual group cohomology. For $p$ an integer  we denote by $_pA$ the subgroup $\{a\in A\mid pa=0\},$ which is the $p$-torsion subgroup of $A$.
\begin{prop} Let $A$ be a $G$-module such that $_2A=0$. Then:
\begin{itemize}
\item[(a)] $ZS^n(G,A)$ is the abelian group of all normalized $n$-cocycles which satisfies $(\ast)$.
\item[(b)] $BS^n(G,A)$ is the abelian group of all normalized $n$-$(\ast)$-boundaries which satisfies $(\ast)$.
\end{itemize}
\end{prop}
\begin{proof}
\begin{itemize}
\item[a)]
Let $\sigma\in ZS^n(G,A)$. Then for any $y_1,\ldots,y_{n+1}\in G$ we have
\begin{equation}\label{eq1}y_1\sigma(y_2,\ldots,y_{n+1})+\sum_{j=1}^n(-1)^j\sigma(y_1,\ldots,y_jy_{j+1},\ldots,y_{n+1})+(-1)^{n+1}\sigma(y_1,\ldots,y_n)=0
\end{equation}
and
\begin{equation}\label{eq2}\sigma(g_1,\ldots,g_n)=-g_1\sigma(g_1^{-1},g_1g_2,g_3,\ldots,g_n)
\end{equation}
\begin{equation}\label{eq3}\sigma(g_1,\ldots,g_n)=-\sigma(g_1,\ldots,g_{i-1}g_i,g_i^{-1},g_ig_{i+1},\ldots,g_n)~~~\text{for}~1<i<n
\end{equation}
\begin{equation}\label{eq4}\sigma(g_1,\ldots,g_n)=-\sigma(g_1,g_2,\ldots,g_{n-1}g_n,g_n^{-1})
\end{equation}
for any $g_1,\ldots,g_n\in G$.

We take $g_n=1$ in (\ref{eq4}) to get $\sigma(g_1,\ldots,g_{n-1},1)=-\sigma(g_1,\ldots,g_{n-1},1)$ and since $_2A=0$ we obtain $\sigma(g_1,\ldots,g_{n-1},1)=0$. Let $g_1=1$ in (\ref{eq2}) and $g_i=1,1<i<n,$ in (\ref{eq3}) to obtain similarly that $\sigma$ is normalized.
We take $g_2=g_1^{-1}$ in (\ref{eq2}) to obtain the first condition of $(\ast)$; we use that $\sigma $ is normalized. In the same way we take $g_{i+1}=g_i^{-1}, 1<i<n,$ in (\ref{eq3}) and $g_{n-1}=1$ in (\ref{eq4}) to obtain the remaining conditions of $(\ast)$.

For the reverse inclusion let $\sigma\in Z^n(G,A)$ such that $(\ast)$ is true. It follows that $\sigma$ is normalized and satisfies (\ref{eq1}). In (\ref{eq1}) we take $$y_1=g_1,y_2=g_1^{-1},y_3=g_1g_2,y_4=g_3,\ldots,y_{n+1}=g_n$$ to obtain (\ref{eq2}).  By taking in (\ref{eq1}) $$y_1=g_1,\ldots,y_{n-1}=g_{n-1}g_n,y_n=g_n^{-1},y_{n+1}=g_n$$ we obtain $(\ref{eq4})$. Next, we fix $i$ such that $1<i<n$ and let
$$y_1=g_1,\ldots,y_i=g_i,y_{i+1}=g_i^{-1},y_{i+2}=g_ig_{i+1},y_{i+3}=g_{i+2},\ldots,y_{n+1}=g_n.$$
It follows that in the sum from (\ref{eq1}) for $j=i-1$ we obtain
$$(-1)^{i-1}\sigma(g_1,\ldots,g_{i-1}g_i,g_i^{-1},g_ig_{i+1},\ldots,g_n),$$
for $j=i+1$ we obtain
$$(-1)^{i+1}\sigma(g_1,\ldots,g_n)$$
and all the other terms are zero, since $\sigma$ is normalized and satisfies $(\ast)$; hence (\ref{eq3}) is true.
\item[b)] Let $\sigma\in BS^n(G,A)$. Then there is $\beta:G^{n-1}\rightarrow A$ such that for any $y_1,\ldots,y_n\in G$ we have
$$\sigma(y_1,\ldots,y_n)=$$
\begin{equation}\label{eq5} y_1\beta(y_2,\ldots,y_n)+\sum_{j=1}^{n-1}(-1)^j\beta(y_1,\ldots,y_jy_{j+1},\ldots,y_n)+(-1)^n\beta(y_1,\ldots,y_{n-1})
\end{equation}
and
\begin{equation}\label{eq6}\beta(g_1,\ldots,g_{n-1})=-g_1\beta(g_1^{-1},g_1g_2,g_3,\ldots,g_{n-1})
\end{equation}
\begin{equation}\label{eq7}\beta(g_1,\ldots,g_{n-1})=-\beta(g_1,\ldots,g_{i-1}g_i,g_i^{-1},g_ig_{i+1},\ldots,g_{n-1})~~~\text{for}~1<i<n-1
\end{equation}
\begin{equation}\label{eq8}\beta(g_1,\ldots,g_{n-1})=-\beta(g_1,g_2\ldots,g_{n-2}g_{n-1},g_{n-1}^{-1})
\end{equation}
for any $g_1,\ldots,g_{n-1}\in G$.
We take $g_1=1$ in (\ref{eq6}), $g_i=1$ in (\ref{eq7}) and $g_{n-1}=1$ in (\ref{eq8}) to obtain that $\beta$ is normalized. Since $\beta$ is normalized and satisfies (\ref{eq6}), (\ref{eq7}), (\ref{eq8}) the same proof as in the first part of a) assure us that $\sigma$ is a $n$-$(\ast)$-boundary. Since $BS^n(G,A)\subseteq ZS^n(G,A)$ we know from a) that $\sigma$ satisfies $(\ast)$.

For the reverse inclusion let $\sigma\in B^n(G,A)$ be a $n$-$(\ast)$-boundary which satisfies $(\ast)$. So there is $\beta:G^{n-1}\rightarrow A$ normalized such that (\ref{eq5}) is true. In (\ref{eq5}) we consider $$y_1=g_1,y_2=g_1^{-1},y_3=g_1g_2,y_4=g_3,\ldots,y_n=g_{n-1}$$
to obtain (\ref{eq6}). To obtain (\ref{eq8}) we take
$$y_1=g_1,y_2=g_2,\ldots,y_{n-2}=g_{n-2}g_{n-1},y_{n-1}=g_{n-1}^{-1},y_n=g_{n-1}$$
in (\ref{eq5}). For the last condition fix $i$ such that $1<i<n-1$ and let
$$y_1=g_1,\ldots,y_{i-1}=g_{i-1}g_i,y_i=g_i^{-1},y_{i+1}=g_i,\ldots, y_n=g_{n-1}.$$
Now, in the sum from (\ref{eq5}) for $j=i-1$ we obtain
$$(-1)^{-1}\beta (g_1,\ldots,g_{n-1}),$$
for $j=i+1$ we obtain
$$(-1)^{i+1}\beta (g_1,\ldots,g_{i-1}g_i,g_i^{-1},g_ig_{i+1},\ldots,g_{n-1})$$
and all the other terms are zero; hence (\ref{eq7}) is true.
\end{itemize}
\end{proof}
For $n=2$ we have the next corollary which shows us that using normalized cochains the conditions to define symmetric cocycles and coboundaries appears to be more easy to remember and to work with them.

\begin{cor} Let $A$ be a $G$-module such that $_2A=0$. Then:
\begin{itemize}
\item[(a)] $ZS^2(G,A)=Z^2(G,A)\cap\{\sigma:G\times G\rightarrow A\mid \sigma(g,g^{-1})=0, \forall g\in G\}$;
\item[(b)] $BS^2(G,A)=B^2(G,A)\cap\{\sigma:G\times G\rightarrow A\mid \sigma(g,g^{-1})=0, \forall g\in G\}$.
\end{itemize}
\end{cor}
\section{Symmetric cohomology and Mackey functors}\label{sec3}
In \cite[Corollary 4.2]{Sin} Singh proves that there is a well-defined restriction and inflation map for continuous symmetric cohomology. Using explicit descriptions we will define a restriction, conjugation and transfer map in algebraic context, for symmetric cohomology. Moreover we will investigate when these maps give a Mackey functor; see \cite[\S 53]{TH}. If $\sigma\in CS^n(G,A)\cap\Ker\partial^n$ is a \emph{symmetric cocycle} we denote by $[\sigma]_S\in\Hs^n(G,A)$ its cohomology class.

\begin{nim}\label{ressym}Let $K\leq H\leq G$ and $n\geq 0$ an integer. It is easy to show that
$$r_H^G:\Hs^n(G,A)\rightarrow\Hs^n(H,A);~~~r_H^G([\sigma]_S)=[\res^G_H(\sigma)]_S$$
is a well-defined linear map and satisfies $r^H_K\circ r^G_H=r^G_K.$
\end{nim}
\begin{lem}\label{lemconjtr}Let $H$ be a subgroup of $G$, $g\in G$, $A$ be a $G$-module and $n\geq 0$ be an integer.
\begin{itemize}
\item[(1)] If $\sigma\in CS^n(H,A)$ then $c_{g,H}(\sigma)\in CS^n({}^gH,A)$ and $\tr_H^G(\sigma)\in CS^n(G,A)$.
\item[(2)] The following two diagrams are commutative
\begin{displaymath}
 \xymatrix{CS^n(H,A)\ar[rr]^{c_{g,H}}\ar[d]^{\partial^n_H} && CS^n({}^gH,A) \ar[d]^{\partial^n_{{}^gH}} \\
                                                     CS^{n+1}(H,A)\ar[rr]^{c_{g,H}}
                                                     &&CS^{n+1}({}^gH,A)
                                            },
\end{displaymath}
\begin{displaymath}
 \xymatrix{CS^n(H,A)\ar[rr]^{\tr_H^G}\ar[d]^{\partial^n_H} && CS^n(G,A) \ar[d]^{\partial^n_G} \\
                                                     CS^{n+1}(H,A)\ar[rr]^{\tr_H^G}
                                                     &&CS^{n+1}(G,A)
                                            }.
\end{displaymath}
\end{itemize}
\end{lem}
\begin{proof}
\begin{itemize}
\item[(1)] We will show the statement for $\tr_H^G$; the similar statement for $c_{g,H}$ is left to the reader. Let $g_1,\ldots,g_n\in H$. We have
\begin{align*}&((1,2)\tr_H^G(\sigma))(g_1,\ldots,g_n)\\&=-g_1\tr_H^G(\sigma)(g_1^{-1},g_1g_2,g_3,\ldots,g_n)\\
&=-g_1\sum_{c\in G/H}\overline{x_2c}\sigma(\overline{x_2c}^{-1}g_1^{-1}\overline{x_1c},\overline{x_1c}^{-1}g_1g_2\overline{x_3c},\overline{x_3c}^{-1}g_3\overline{x_4c},\ldots,\overline{x_nc}^{-1}g_n\overline{c})\\
&=\sum_{c\in G/H}-\overline{x_1c}~\overline{x_1c}^{-1}g_1\overline{x_2c}\sigma((\overline{x_1c}^{-1}g_1\overline{x_2c})^{-1},\overline{x_1c}^{-1}g_1\overline{x_2c}~\overline{x_2c}^{-1}g_2\overline{x_3c},\ldots,\overline{x_nc}^{-1}g_n\overline{c})\\
&=\sum_{c\in G/H}\overline{x_1c}((1,2)\sigma)(\overline{x_1c}^{-1}g_1\overline{x_2c},\overline{x_2c}^{-1}g_2\overline{x_3c},\ldots, \overline{x_nc}^{-1}g_n\overline{c})\\&=\tr_H^G(\sigma)(g_1,\ldots,g_n),
\end{align*}
where the last equality is true since $\sigma\in CS^n(H,A)$.

Let $i$ be an integer such that $1<i<n$.
\begin{align*}
&((i,i+1)\tr_H^G(\sigma))(g_1,\ldots,g_n)\\&=-\tr_H^G(\sigma)(g_1,\ldots, g_{i-1}g_i,g_i^{-1},g_ig_{i+1},\ldots,g_n)\\
&=-\sum_{c\in G/H}\overline{x_1c}\sigma(\overline{x_1c}^{-1}g_1\overline{x_2c},\ldots,\overline{x_{i-1}c}^{-1}g_{i-1}g_i\overline{x_{i+1}c},\\
&\overline{x_{i+1}c}^{-1}g_i^{-1}\overline{x_ic},\overline{x_ic}^{-1}g_ig_{i+1}\overline{x_{i+2}c}\ldots,\overline{x_nc}^{-1}g_n\overline{c})\\
&=\sum_{c\in G/H}\overline{x_1c}((i,i+1)\sigma)(\overline{x_1c}^{-1}g_1\overline{x_2c},\overline{x_2c}^{-1}g_2\overline{x_3c},\ldots, \overline{x_nc}^{-1}g_n\overline{c})\\&=\tr_H^G(\sigma)(g_1,\ldots,g_n),
\end{align*}
where the last equality is true since $\sigma\in CS^n(H,A)$.
Finally for the action of $(n,n+1)$ we obtain
\begin{align*}
&((n,n+1)\tr_H^G(\sigma))(g_1,\ldots,g_n)\\&=-\tr_H^G(\sigma)(g_1,\ldots,g_{n-1}g_n,g_n^{-1})\\
&=-\sum_{c\in G/H}\overline{x_1g_n^{-1}c}\sigma(\overline{x_1g_n^{-1}c}^{-1}g_1\overline{x_2g_n^{-1}c},\ldots,\overline{x_{n-1}g_n^{-1}c}^{-1}g_{n-1}g_n\overline{g_n^{-1}c},\overline{g_n^{-1}c}^{-1}g_n^{-1}\overline{c})\\
\end{align*}
If $c$ runs in $G/H$ then $g_n^{-1}c$ runs in $G/H$ hence we can replace $g_n^{-1}c$ by $c$ and we  continue the above equalities:
\begin{align*}
&((n,n+1)\tr_H^G(\sigma))(g_1,\ldots,g_n)\\&=-\sum_{c\in G/H}\overline{x_1c}\sigma(\overline{x_1c}^{-1}g_1\overline{x_2c},\ldots,\overline{x_{n-1}c}^{-1}g_{n-1}g_n\overline{c},\overline{c}^{-1}g_n^{-1}\overline{g_nc})
\\&=-\sum_{c\in G/H}\overline{x_1c}\sigma(\overline{x_1c}^{-1}g_1\overline{x_2c},\ldots,\overline{x_{n-1}c}^{-1}g_{n-1}\overline{g_nc}~\overline{g_nc}^{-1}g_n\overline{c},(\overline{g_nc}^{-1}g_n\overline{c})^{-1})
\\&=\sum_{c\in G/H}\overline{x_1c}(n,n+1)\sigma(\overline{x_1c}^{-1}g_1\overline{x_2c},\ldots,\overline{x_{n-1}c}^{-1}g_{n-1}\overline{x_nc},\overline{x_nc}^{-1}g_n\overline{c})
\\&=\sum_{c\in G/H}\overline{x_1c}\sigma(\overline{x_1c}^{-1}g_1\overline{x_2c},\ldots,\overline{x_{n-1}c}^{-1}g_{n-1}\overline{x_nc},\overline{x_nc}^{-1}g_n\overline{c})
\\&=\tr_H^G(\sigma)(g_1,\ldots,g_n),
\end{align*}
where the fourth equality is true since $\sigma\in CS^n(H,A)$.

\item[(2)] The proof of that statement (2) follows from the similar result for
the usual group cohomology and statement (1).
\end{itemize}
\end{proof}

Now Lemma \ref{lemconjtr} assure us that there are well-defined linear maps, conjugation and transfer, which we consider in the next definition.
\begin{defn}\label{defconjtr} The \emph{conjugation map} for symmetric cohomology is
$$c_{g,H}:\Hs^n(H,A)\rightarrow \Hs^n(^gH,A),~~~c_{g,H}([\sigma]_S)=[c_{g,H}(\sigma)]_S,$$
and the \emph{transfer map} is
$$t_H^G:\Hs^n(H,A)\rightarrow\Hs^n(G,A),~~~t_H^G([\sigma]_S)=[\tr_H^G(\sigma)]_S,$$
where $\sigma\in CS^n(H,A)$ is a symmetric cocycle.
\end{defn}
\begin{rem} It might be possible that $c_{g,H}$ and $t_H^G$ to be defined as consequences of \cite[Proposition 4.1]{Sin} but we prefer these explicit constructions.
\end{rem}
It is well known that there is a natural map from the symmetric group cohomology to the usual group cohomology  $$i:\Hs^n(G,A)\rightarrow \Hc^n(G,A),~~~ i([\sigma]_S)=[\sigma].$$
\begin{thm}\label{thmmackey} Let $n\geq 0$ be an integer, let $K,H$ be subgroups of $G$ with $K\leq H$ and $g\in G$. If the natural map $i$ is one to one then the family of abelian groups $\{\Hs^n(H,A)\}_{H\leq G}$ together with the linear maps $\{r_K^H,t_K^H,c_{g,H}\}_{K\leq H,g\in G}$ is a Mackey functor.
\end{thm}
\begin{proof} We recall the well-known axioms of a Mackey functor \cite[$\S 53$]{TH} which we want to prove
\begin{itemize}
\item[(i)] $r_L^Kr_K^H=r_L^H;~~~t_K^Ht_L^K=t_L^H$ if $L\leq K\leq H;$
\item[(ii)] $r_H^H=t_H^H=id_{\Hs^n(H,A)}$;
\item[(iii)] $c_{gh,H}=c_{g,^hH}c_{h,H}$ if $g,h\in G$;
\item[(iv)]$c_{h,H}=id_{\Hs^n(H,A)}$ if $h\in H$;
\item[(v)] $c_{g,K}r_K^H=r^{^gH}_{^gK}c_{g,H},~~~c_{g,H}t_K^H=t_{^gK}^{^gH}c_{g,K}$ if $K\leq H$ and $g\in G$;
\item[(vi)] Mackey axiom: if $L,K\leq H$ then
$$r_L^Ht_K^H=\sum_{h\in[L\backslash H /K]}t_{L\cap^hK}^Lr_{L\cap^hK}^{^hK}c_{h,K}.$$
\end{itemize}
The first part of statement (i), statements (ii),(iii) and (v) are easy to verify since the equalities hold for maps of cochain complex in general, the hypothesis that $i$ is injective. We exemplify with (v). Let $\sigma\in CS^n(H,A)$ and $g_1,\ldots,g_n\in K$ and $g\in G$. We have
$$c_{g,K}(\res^H_K(\sigma))({}^gg_1,\ldots,{}^gg_n)=g\res^H_K(g_1,\ldots,g_n)=g\sigma(g_1,\ldots,g_n);$$
$$\res^{{}^gH}_{{}^gK}(c_{g,H}(\sigma))({}^gg_1,\ldots,{}^gg_n)=c_{g,K}(\sigma)({}^gg_1,\ldots,{}^gg_n)=g\sigma(g_1,\ldots,g_n).$$
For the second part let $\sigma\in CS^n(K,A)$ and ${}^gg_1,\ldots,{}^gg_n\in{}^gH$. By abuse of notation we use the same notations for $t_K^H$ as in \ref{tran}.
\begin{align*}
c_{g,H}(\tr_K^H(\sigma))&({}^gg_1,\ldots,{}^gg_n)\\&=g\tr_K^H(g_1,\ldots,g_n)\\&=g\sum_{c\in H/K}\overline{x_1c}\sigma(\overline{x_1c}^{-1}g_1\overline{x_2c},\overline{x_2c}^{-1}g_2\overline{x_3c},\ldots,\overline{x_nc}^{-1}g_n\overline{c})
\end{align*}
It is clear that if $H=\bigcup_{c\in H/K}c$ then ${}^gH=\bigcup_{{}^gc\in{}^gH/{}^gK}{}^gc$ and $\overline{{}^gc}={}^g\overline{c}$ for any $c\in H/K$. Hence we have
\begin{align*}
\tr_{{}^gK}^{{}^gH}(c_{g,K}(\sigma))&({}^gg_1,\ldots,{}^gg_n)\\&=\sum_{{}^gc\in ^gH/{}^gK}\overline{{}^gx_1{}^gc}~~c_{g,K}(\sigma)(\overline{{}^gx_1{}^gc}^{-1}{}^gg_1\overline{{}^gx_2{}^gc},\ldots,\overline{{}^gx_n{}^gc}^{-1}{}^gg_n\overline{{}^gc})\\&=
\sum_{^gc\in^gH/{}^gK}g\overline{x_1c}g^{-1}g\sigma(\overline{x_1c}^{-1}g_1\overline{x_2c},\ldots,\overline{x_nc}^{-1}g_n\overline{c})\\&=
g\sum_{c\in H/K}\overline{x_1c}\sigma(\overline{x_1c}^{-1}g_1\overline{x_2c},\ldots,\overline{x_nc}^{-1}g_n\overline{c})
\end{align*}
For (iv) we need to prove that $[c_{h,H}(\sigma)]_S=[\sigma]_S$ if $h\in H$ and $\sigma$ is a symmetric cocycle. Since for the usual group cohomology we have that $[c_{h,H}(\sigma)]=[\sigma]$, using the injectivity and the definition of $i$  we are done. With some technical adjustments similar arguments work for (vi) and the second part of (i), again using the injectivity of the map $i$.
\end{proof}

By \cite[Proposition 4.1]{Starc} and Theorem \ref{thmmackey} we have the next corollary.
\begin{cor}
Let $n\geq 0$ be an integer , let $K,H$ be subgroups of $G$ with $K\leq H$ and $g\in G$. Let $A$ be a $G$-module such that $n+1$ is not a zero divisor and the equation $n!x=a$ has a unique solution. Then the family of abelian groups $\{\Hs^n(H,A)\}_{H\leq G}$ together with the linear maps $\{r_K^H,t_K^H,c_{g,H}\}_{K\leq H,g\in G}$ is a Mackey functor.
\end{cor}
\begin{rem}\label{remopenpb}We observe from the proof of Theorem \ref{thmmackey} that some axioms of the definition of the Mackey functor for symmetric cohomology are proved in general without the assumption that $i$ is injective. It is an open problem if all the axioms can be proved in general, without this assumption.
\end{rem}
\textbf{Acknowledgements.} The author is grateful to the referee for his kind and valuable comments which contribute to major improvements of a previous version of this article.



\begin{thebibliography}{00}

\bibitem{Sin} M. Singh, Symmetric continuous cohomology of topological groups,	 Homology Homot. Appl., \textbf{15}(1) (2013), 279-302.
\bibitem{St3D} M. Staic, From $3$-algebras to $\Delta$-groups and symmetric cohomology, J. Algebra \textbf{322} (2009),1360--1378.
\bibitem{Starc} M. Staic, Symmetric cohomology of groups in low dimension, Arch. Math \textbf{93} (2009), 205--211.
\bibitem{We} E. Weiss, Cohomology of groups, Academic Press, New York, 1969.
\bibitem{TH}J. Th\'{e}venaz,  G-Algebras and Modular Representation Theory, Clarendon Press, Oxford, 1995.
\end{thebibliography}



\end{document}